\documentclass[10pt, reqno]{amsart}
\usepackage{amssymb,amsmath,amsthm,color, cite}
\theoremstyle{plain}
\newtheorem{theorem}{Theorem}[section]
\newtheorem{proposition}[theorem]{Proposition}

\newtheorem{lemma}[theorem]{Lemma}
\newtheorem{corollary}[theorem]{Corollary}

\theoremstyle{remark}
\newtheorem{remark}[theorem]{Remark}

\usepackage
[bookmarks=true,
   bookmarksnumbered=false,
   bookmarkstype=toc]
{hyperref}

\numberwithin{equation}{section}

\newcommand{\C}{\mathbb{C}}
\newcommand{\R}{\mathbb{R}}

\newcommand{\F}{\mathcal{F}}

\newcommand{\und}[1]{\underline{#1}}

\newcommand{\1}{\textbf{1}}

\newcommand{\Fr}{\text{Fr}\,}

\def\({\left(}
\def\){\right)}
\def\<{\left\langle}
\def\>{\right\rangle}

\newcommand{\eps}{\varepsilon}

\newcommand{\qtq}[1]{\quad\text{#1}\quad}
\newcommand{\sgn}{\text{sgn}}

\newcommand{\rre}{\mathbb{R}}

\begin{document}
\title[NLS with delta potential]{Modified scattering for the $1d$ cubic NLS with a repulsive delta potential}

\author[S. Masaki]{Satoshi Masaki}
\address{Department of Systems Innovation \\
Graduate School of Engineering Science \\
Toyonaka, Osaka, Japan}
\email{masaki@sigmath.es.osaka-u.ac.jp}

\author[J. Murphy]{Jason Murphy}
\address{Department of Mathematics and Statistics \\
Missouri University of Science and Technology \\
Rolla, MO, USA}
\email{jason.murphy@mst.edu}

\author[J. Segata]{Jun-ichi Segata}
\address{Mathematical Institute, Tohoku University\\
6-3, Aoba, Aramaki, Aoba-ku, Sendai 980-8578, Japan}
\email{segata@m.tohoku.ac.jp}



\maketitle

\begin{abstract}  We consider the initial-value problem for the $1d$ cubic nonlinear Schr\"odinger equation with a repulsive delta potential.  We prove that small initial data in a weighted Sobolev space lead to global solutions that decay in $L^\infty$ and exhibit modified scattering.
\end{abstract}


%
%
\section{Introduction}

We consider the initial-value problem for the cubic nonlinear Schr\"odinger equation (NLS) with a repulsive delta potential in one space dimension.  This equation takes the following form:
\begin{equation}\label{nls}
\begin{cases}
i\partial_t u = Hu + \lambda |u|^2 u, \\
u(0) = u_0.
\end{cases}
\end{equation}
Here $u:\R_t\times\R_x\to\C$, $\lambda\in\R$, and $H$ is the Schr\"odinger operator
\[
H = -\tfrac12\partial_x^2 + q\delta(x),
\]
where $q>0$ (the repulsive case) and $\delta$ is the Dirac delta distribution.  We give a more detailed introduction to the operator $H$ in Section~\ref{S:linear}.

Our goal is to describe the long-time decay and asymptotics of small solutions to \eqref{nls}.  As is well-known, the free $1d$ cubic NLS (i.e. \eqref{nls} with $H=-\tfrac12\partial_x^2$) is a borderline case: solutions with small initial data decay in $L^\infty$ at the sharp rate (i.e. matching linear solutions), but to describe the asymptotics one must incorporate a phase correction to the usual linear behavior \cite{DeiftZhou, HN1, IT, KP, LS}.  Recent works \cite{GPR, N, N2} have established similar `modified scattering' results for the case of the $1d$ cubic NLS with potentials of sufficient regularity and decay.  In particular, these results do not apply to the case of a delta potential.  

Our main result is the following:
\begin{theorem}\label{T} Fix $q>0$ and $\beta\in(0,\frac18)$.  Let $u_0\in \Sigma:=\{f\in H^1(\R):xf\in L^2(\R)\},$ with
\[
\|u_0\|_{\Sigma} = \eps. 
\]  
For $\eps>0$ sufficiently small, there exists a unique solution 
$u\in C([0,\infty);H^{1}(\R))$ 
to \eqref{nls} with $u(0)=u_0$.  This solution obeys the following decay estimate:
\begin{equation}\label{E:decay}
\|u(t)\|_{L^\infty(\R)} \lesssim \eps\langle t\rangle^{-\frac12} \qtq{for all}t\geq0. 
\end{equation}
Furthermore, there exists a unique $W\in L^\infty(\R)$ so that 
\begin{equation}\label{E:asymptotics}
\bigl\| u(t) - (it)^{-\frac12}e^{\frac{ix^2}{2t}} W(\tfrac{x}{t})e^{-i
\lambda|W(\frac{x}{t})|^2\log t}\bigr\|_{L^\infty(\R)} \lesssim t^{-\frac34+\beta}\qtq{as}t\to\infty.
\end{equation}
\end{theorem}

\begin{remark} The asymptotics appearing in \eqref{E:asymptotics} have the same form as those for the free 1d cubic NLS.  This phenomenon was already observed by the third author \cite{S}, who constructed solutions to \eqref{nls} scattering to prescribed final states in the sense of \eqref{E:asymptotics}.
\end{remark}

Before discussing the proof of Theorem~\ref{T}, we mention several related results.  As already remarked above, the recent works \cite{GPR, N, N2} consider the $1d$ cubic NLS with regular and decaying potentials and establish modified scattering.  The problem of (linear) scattering for larger power nonlinearities with regular and decaying potentials was also studied in \cite{CGV}.  

For the NLS with a delta potential (i.e. (\ref{nls}) with nonlinear term $\lambda|u|^{p}u$), there are several relevant results.  For the case of a defocusing nonlinearity and repulsive delta potential, scattering in the energy space in the mass-supercritical regime (i.e. $p>4$) was proven in \cite{BV} via the concentration compactness approach.  The paper \cite{II} studied the dynamics below the ground state for a focusing mass-supercritical nonlinearity and repulsive delta potential, also via concentration compactness.  The series of works \cite{DH, HMZ1, HMZ2, HZ} investigated the interesting scenario of  solitons interacting with the delta potential.  Asymptotic stability of the nonlinear ground state under even perturbations was also studied in \cite{HZ2, DP}; in particular, the authors of \cite{DP} utilized complete integrability and inverse scattering techniques to establish precise asymptotics.  Finally, as mentioned above, in \cite{S} the third author constructed solutions scattering to prescribed final states in the sense of \eqref{E:asymptotics}. 

\subsection{Strategy of the proof}  To prove Theorem~\ref{T}, we adapt the strategy of Naumkin \cite{N}, which is in fact similar to the strategy of Hayashi and Naumkin \cite{HN1} originally used to study the free NLS.  

Solutions to the linear Schr\"odinger equation with a delta potential are generated by a unitary group $U(t)$, which is diagonalized by the `distorted' Fourier transform $\F_q$ (see Section~\ref{S:linear}):
\[
U(t) = \F_q^{-1} e^{-it\xi^2/2}\F_q. 
\]
To study the asymptotics of a solution $u(t)$, one can work with the variable
\[
w(t) =\F_q U(-t) u(t). 
\]
In terms of the original variable, one can compute that
\[
u(t) = M(t) D(t) V(t) w(t)
\]
for a certain operator $V$, where $M$ and $D$ are the familiar modulation and dilation operators.  Thus if one understands the asymptotics of $w$ (and $V$),  one can describe the asymptotics of $u$.  The advantage of working with $w$ is that if $u$ solves \eqref{nls}, then $w$ solves an ordinary differential equation, namely, 
\begin{equation}\label{ode}
i\partial_t w = \lambda t^{-1} V(t)^{-1}\bigl[|V(t)w|^2 V(t)w\bigr]. 
\end{equation}

For the free NLS, one computes $V=\F_0 M\F_0^{-1}=e^{-\frac{i}{2t}\partial_x^2}$, where $\F_0$ is the ordinary Fourier transform. Hence $V,V^{-1}$ simply behave like the identity as $t\to\infty$, and \eqref{ode} may be approximated by
\begin{equation}\label{ode2}
i\partial_t w = \lambda t^{-1} |w|^2 w + \mathcal{O}(t^{-1-})
\end{equation}
as $t\to\infty$ (under suitable bootstrap assumptions on $w$).  As solutions to \eqref{ode2} remain bounded, one can obtain boundedness of $w(t)$ and hence the desired decay for $u(t)$. One can also use the above equation to deduce the asymptotic behavior.

In \cite{N}, Naumkin relied on regularity and decay assumptions on the potential in order to prove bounds and asymptotics for the operators for $V$ and $V^{-1}$.  The asymptotics, which are given in terms of the transmission and reflection coefficients for the potential, depend both on the input function as well as its reflection.  Thus one finds that under suitable bootstrap assumptions on $w$, \eqref{ode} may be approximated by a $2\times 2$ system of the form
\begin{equation}\label{ode3}
i\partial_t \vec w = \lambda t^{-1} A(w,x) \vec w + \mathcal{O}(t^{-1-}), \quad \vec w(t,x)=(w(t,x),w(t,-x)). 
\end{equation}
One finds that the matrix $A$ is hermitian, and hence solutions to \eqref{ode3} once again remain bounded.  As in the free case, one can also use the equation to deduce the asymptotic behavior. 

In this paper, we treat the case of a repulsive delta potential, which does not belong to the class of potentials treated in previous works.  The key point is that in this special case, we are able to compute many of the relevant quantities fairly explicitly.  In the end, we are also led to an approximate equation for $w$ of the form \eqref{ode3} with a hermitian matrix $A$  (see Proposition~\ref{P:ode}).  To establish the asymptotic behavior, we note that upon diagonalizing $A$ with a unitary matrix, the equation \eqref{ode3} reduces to a diagonal system in which both equations have the form of \eqref{ode2}.  From this point, it is straightforward to compute the asymptotic behavior.

The rest of this paper is organized as follows:  In Section~\ref{S:linear}, we collect some basic facts about the linear Schr\"odinger equation with a delta potential, including some results about the distorted Fourier transform.  In Section~\ref{S:V}, we study the operators $V(t)$ and $V(t)^{-1}$ introduced above.  In particular, we establish $L^\infty$ asymptotics and $\dot H^1$ bounds for these operators.  In Section~\ref{S:decay},  we prove the first part of Theorem~\ref{T}, namely the $L^\infty$ decay estimate \eqref{E:decay}.  Finally, in Section~\ref{S:asymptotics}, we establish the asymptotics \eqref{E:asymptotics}. 

\subsection*{Notation}  We write $A\lesssim B$ to denote $A\leq CB$ for some $C>0$.  The usual modulation and dilation operators are defined as follows:
\begin{equation}\label{MD}
[M(t) f](x) = e^{\frac{ix^2}{2t}}f(x),\quad [D(t)f](x) = (it)^{-\frac12}f(\tfrac{x}{t}). 
\end{equation}
We write
\[
\int_+ f(x)\,dx = \int_0^\infty f(x)\,dx\qtq{and}\int_-f(x)\,dx=\int_{-\infty}^0f(x)\,dx. 
\]
We define 
\[
\1(x) = \begin{cases} 1 & x>0 \\ 0 & x<0. \end{cases}
\]
We use the notation $\und x = -x$, $\und \xi = -\xi$, and so on. We write $\und f(x) = f(\und x)$.  For time-dependent functions, we will also use $\und f(t,x) = f(t,\und x)$; i.e. an underline denotes reflection in the spatial variables only.   

\subsection*{Acknowledgements} 
S.M. was supported by the Sumitomo Foundation Basic Science Research
Projects No.\ 161145 and by JSPS KAKENHI Grant Numbers JP17K14219, JP17H02854, and JP17H02851.
J.M. was supported in part by NSF DMS-1400706. 
J.S. was partially supported by JSPS KAKENHI Grant Number JP17H02851.


\section{The linear Schr\"odinger equation with a repulsive delta potential}\label{S:linear}

The linear Schr\"odinger equation with a delta potential is a classical and well-understood model from quantum mechanics.  We refer the reader to \cite{AG} for a comprehensive background.  Throughout the paper, we restrict our attention to the case of a repulsive delta potential, i.e.
\[
H = -\tfrac12\partial_x^2 + q\delta(x)\qtq{with}q>0.
\]

More precisely, the domain of $H$ is defined to be
\[
\{u\in H^1(\R)\cap H^2(\R\backslash\{0\}): \partial_xu(0+)-\partial_x u(0-)=2qu(0)\},
\]
where $\pm$ denote limits from the right or left, and $H=-\tfrac12\partial_x^2$ on its domain. Then $H$ extends to a self-adjoint operator on $L^2$ with a purely absolutely continuous essential spectrum equal to $[0,\infty)$. If $q<0$, then $H$ has one simple negative eigenvalue; this is the \emph{attractive} case.  If $q>0$, then $H$ has no eigenvalues; this is the \emph{repulsive} case and it is the case under consideration in this paper.

The \emph{Jost functions} for $H$ are the solutions $f_\pm = f_\pm(x,\xi)$ to the eigenvalue problem
\[
Hf = \tfrac12 \xi^2 f\qtq{such that}\lim_{x\to\pm\infty}[f_\pm(x,\xi)-e^{\pm ix\xi}] = 0.
\]
For the delta potential, these functions may be written down explicitly.  To do so, one introduces the \emph{transmission and reflection coefficients}, defined by
\begin{equation}\label{TR}
T(\xi) =  \tfrac{i\xi}{i\xi-q} \qtq{and} R(\xi) = \tfrac{q}{i\xi-q},\qtq{where}\xi\in\R.
\end{equation}
The Jost functions are then given by
\begin{equation}\label{j1}
f_+(x,\xi) = \1(x)e^{ix\xi} +  \1(\und x)[\tfrac{1}{T(\xi)}e^{ix\xi} + \tfrac{R(\xi)}{T(\xi)}e^{-ix\xi}]
\end{equation}
and
\begin{equation}\label{j2}
f_-(x,\xi)= \1(\und x) e^{-ix\xi} + \1(x)[\tfrac{1}{T(\xi)}e^{-ix\xi} + \tfrac{R(\xi)}{T(\xi)}e^{ix\xi}]
\end{equation}
for any $\xi\neq 0$.  We extend $f_\pm$ continuously to $x=0$ by setting $f_\pm(0,\xi)=1$. 

Noting that $f_-(x,\xi)=f_+(\und x,\xi)$, we simply write
\[
f_+(x,\xi)=f(x,\xi)
\]
and work with this function henceforth. 

The Jost functions may be used to define a \emph{distorted Fourier transform} for $H$, which may be used to diagonalize and solve the linear equation.  In particular, we define
\begin{equation}\label{Fq}
\F_q\phi(\xi) = \tfrac{1}{\sqrt{2\pi}}\int_\R \overline{K(x,\xi)}\phi(x)\,dx, 
\end{equation}
where the kernel $K(\cdot,\cdot)$ is defined by
\begin{equation}\label{def:K}
K(x,\xi) = \begin{cases}
\1(\xi)T(\xi)f(x,\xi) + \1(\und\xi)T(\und\xi)f(\und x,\und\xi) & \xi\neq 0 \\  0 & \xi=0.
\end{cases}
\end{equation}
Note $K(x,\cdot)$ is continuous at $\xi=0$ provided $q>0$.

As we will discuss below, $\F_q$ is unitary on $L^2(\R)$ for $q>0$, with the inverse transform given by
\begin{equation}\label{def:FqI}
\F_q^{-1}\psi(x) = \tfrac{1}{\sqrt{2\pi}} \int_\R K(x,\xi)\psi(\xi)\,d\xi. 
\end{equation}

As one can check, the solution to the linear Schr\"odinger equation
\[
\begin{cases}
i\partial_t u = Hu \\
u(0)=\phi 
\end{cases}
\]
is given by
\[
u(t)=U(t)\phi,\qtq{where} U(t):=\F_q^{-1}e^{-\frac{it}{2}\xi^2}\F_q. 
\]

 Note that we may write
\[
K(x,\xi)=2P_e\bigl[\1(\xi)T(\xi)f(x,\xi)\bigr],
\]
where $P_e$ denotes the projection to even functions in $(x,\xi)$.  Using the definition of $f$ and the fact that $P_e[F(\und x,\xi)]=P_e[F(x,\und\xi)],$ we may also write
\begin{equation}\label{ke2} 
\begin{aligned}
K(x,\xi)&=2P_e\bigl[\1(\xi)\bigl\{\1(x)T(\xi)e^{ix\xi}+\1(\und x)[e^{ix\xi}+R(\xi)e^{-ix\xi}]\bigr\}\bigr] \\
& =2P_e\bigl[\1(x)\{\1(\xi)T(\xi)e^{ix\xi}+\1(\und\xi)[e^{ix\xi}+R(\und\xi)e^{-ix\xi}]\bigr\}\bigr].
\end{aligned}
\end{equation}

In the rest of this section, we collect some basic properties of the distorted Fourier transform.  

\subsection{The distorted Fourier transform} As $K(\cdot,\cdot)$ is bounded, the formulas in \eqref{Fq} and \eqref{def:FqI} make sense pointwise for functions in $L^1$; in fact, by the dominated convergence theorem, they define continuous functions.  The operator $\F_q$ extends to a unitary operator on $L^2$, with inverse given by $\F_q^{-1}$. This was shown, for example, by the third author in \cite{S} by representing $\F_q$ and $\F_q^{-1}$ in terms of the standard Fourier transform; we refer the reader to \cite[Section~3]{S} for details.

We record two further identities for $\F_q$ and $\F_q^{-1}$ that aid in establishing some boundedness properties.
\begin{lemma}\label{L:IDs} \text{ }

For $\phi\in L^1(\R)$ and $\xi\in\R$, 
\begin{equation}\label{E:FID}
\F_q\phi(\xi) = \F_0\phi(\xi) + \tfrac{1}{\sqrt{2\pi}}\bar R(|\xi|)\int_\R e^{-i|x||\xi|} \phi(x)\,dx.
\end{equation}

For $\psi\in L^1(\R)$ and $x\in\R$, 
\begin{equation}\label{E:FID2}
\F_q^{-1}\psi(x) = \F_0^{-1}\psi(x) + \tfrac{1}{\sqrt{2\pi}} \int_\R e^{i|x||\xi|}R(|\xi|)\psi(\xi)\,d\xi.
\end{equation}
\end{lemma}

\begin{proof} These identities follow from direct computation.  For example, suppose $\xi>0$.  Then recalling \eqref{def:K} and \eqref{j1}, 
\begin{align*}
\int_\R \overline{K(x,\xi)}\phi(x)\,dx & = \int_+ \bar T(\xi) e^{-ix\xi}\phi(x)\,dx + \int_- [e^{-ix\xi}+\bar R(\xi)e^{ix\xi}]\phi(x)\,dx  \\
& = \int_\R e^{-ix\xi}\phi(x) \,dx + \bar R(\xi)\int_\R e^{-i|x|\xi}\phi(x)\,dx,
\end{align*}
where we use $R=T-1$ in the last line.  Proceeding in this way leads to the identities above. \end{proof}

\begin{corollary} The following hold:
\begin{align}
[\F_q\phi](0) & = 0\qtq{whenever}\langle x\rangle\phi\in L^2, \label{E:F00} \\
\|\partial_\xi \F_q\phi\|_{L_\xi^2}& \lesssim \|\langle x\rangle \phi\|_{L^2}, \label{E:DF} \\
\| \xi\F_q\phi\|_{L_\xi^2}& \lesssim |\phi(0)|+\|\partial_x \phi\|_{L^2}\lesssim \|\phi\|_{H^1}, \label{E:XF} \\
\|x\F_q^{-1}\psi\|_{L^2} & \lesssim \|\psi\|_{H^1} \qtq{whenever}\psi(0)=0.  \label{E:XF1}
\end{align}
In particular, $\F_q$ is a bijection from $\langle x \rangle^{-1} L^2(\R)$ to $\{f\in H^1 : f(0)=0\}$ or
from $\Sigma$ to $\{f \in \Sigma : f(0) = 0\}$. 
\end{corollary}

\begin{proof} For \eqref{E:F00}, first note that $\langle x\rangle^{-1}L^2(\R)\subset L^1(\R)$ by Cauchy--Schwarz.  Thus \eqref{E:F00} follows immediately from \eqref{E:FID} and the fact that $R(0)=-1$.

Next, the estimate \eqref{E:DF} follows from a direct computation using \eqref{E:FID}.  

Let us show \eqref{E:XF}. An integration by parts shows
\[
|\xi|\int_\R e^{-i|x||\xi|}\phi(x)\,dx
= -2i\phi(0) - i \int_\R e^{-i|x||\xi|}\sgn(x)\partial_x\phi(x)\,dx,
\]
where $\sgn$ is the signum function.  Combining this with \eqref{E:FID}, the fact that $R\in L^{2}\cap L^{\infty}$, and Sobolev embedding, we have \eqref{E:XF}.





Finally, an integration by parts shows
\[
|x|\int_\R e^{i|x||\xi|}R(|\xi|)\psi(\xi)\,d\xi = -2i\psi(0) + i \int_\R e^{i|x||\xi|}\sgn(\xi)\partial_\xi[R(|\cdot|)\psi](\xi)\,d\xi.
\]
Combining this with \eqref{E:FID2}, we deduce \eqref{E:XF1}.  Note that $\psi(0)$ is well-defined for any $\psi\in H^1$ thanks to the Sobolev embedding $H^1(\R)\subset C(\R)$. \end{proof}


\section{The operators $V(t)$ and $V^{-1}(t)$}\label{S:V}

In analogy with the free linear Schr\"odinger equation, we define an operator $V(t)$ by imposing
\[
U(t) = M(t) D(t) V(t) \F_q,
\]
where $M$ and $D$ are as in \eqref{MD}.  In particular, 
\[
V(t)^{-1}=e^{\frac{it}{2}\xi^2}\F_q M(t)D(t).
\]
As mentioned in the introduction, in the case of the free Schr\"odinger equation one has $V(t) = e^{-\frac{i}{2t}\partial_x^2}$.  We can understand the effect of adding the delta potential by understanding the resulting changes to the operators $V(t)$ and $V(t)^{-1}$.   

One can see immediately that $V(t)$ and $V^{-1}(t)$ are bounded on $L^2$.  One can also write down the integral kernels for these operators explicitly:

\begin{lemma}\label{V-ik} The operators $V(t)$ and $V^{-1}(t)$ are given by the following: 
\begin{align*}
[V(t)\psi](x) & =\int_\R  \sqrt{\tfrac{it}{2\pi}}e^{-\frac{it}{2}(x^2+\xi^2)}K(tx,\xi)\psi(\xi)\,d\xi, \\
[V^{-1}(t)\phi](\xi) &=\int_\R \overline{  \sqrt{\tfrac{it}{2\pi}}e^{-\frac{it}{2}(x^2+\xi^2)}K(tx,\xi)}\phi(x)\,dx,
\end{align*}
where $K(\cdot,\cdot)$ is as in \eqref{def:K}. 
\end{lemma}

\begin{proof} To deduce the first identity, one can write $\F_q^{-1}e^{-\frac{it}{2}\xi^2}=M(t)D(t)V(t)$, expand the left-hand side, and force $M(t)D(t)$ to appear.  For the second identity, one can write $V(t)^{-1}=e^{\frac{it}{2}\xi^2}\F_q M(t)D(t)$, expand the right-hand side, and change variables. 
\end{proof}

As described in the introduction, it will be essential to understand the asymptotics of $V(t)$ and $V^{-1}(t)$ as $t\to\infty$.  

We begin with $V(t)$.  We put the kernel into a form amenable to a stationary phase type analysis:  Recalling the second expression in \eqref{ke2}, let $t>0$ and introduce the function $a_t(x,\xi)$ as follows: 
\begin{equation}\label{at}
\begin{aligned} 
&\sqrt{\tfrac{it}{2\pi}}e^{-\frac{it}{2}(x^2+\xi^2)}K(tx,\xi) \\
& = 2\sqrt{\tfrac{it}{2\pi}}P_e\bigl[e^{-\frac{it}{2}(x^2+\xi^2)}\1(x)\bigl\{\1(\xi)T(\xi)e^{itx\xi}+\1(\und\xi)[e^{itx\xi}+R(\und\xi)e^{-itx\xi}]\bigr\}\bigr] \\
& = 2\sqrt{\tfrac{it}{2\pi}}P_e\bigl[ \1(x)\bigl\{ e^{-\frac{it}{2}(x-\xi)^2}[\1(\xi)T(\xi)+\1(\und\xi)] + e^{-\frac{it}{2}(x+\xi)^2}\1(\und\xi)R(\und\xi)\bigr\}\bigr]\\
& =: 2P_e[a_t(x,\xi)]. 
\end{aligned}
\end{equation}
Then, after a change of variables in the second integral,
\begin{equation}\label{V1}
[V(t)\psi](x) = \int_\R a_t(x,\xi)\psi(\xi) \,d\xi + \int_\R a_t(\und x,\xi)\psi(\und\xi)\,d\xi. 
\end{equation}

To prove asymptotics will require the evaluation and estimation of a few special integrals.  For convenience, we record the results we need in the following lemma.

\begin{lemma} The following hold:
\begin{equation}\label{g-e}
\int_\R e^{-\frac{ix^2}{2}}\,dx = \sqrt{\tfrac{2\pi}{i}}.
\end{equation}
Denoting
\begin{equation}\label{Fr}
\Fr(y) = \sqrt{\tfrac{i}{2\pi}}\int_{-\infty}^{-y} e^{-\frac{ix^2}{2}}\,dx,
\end{equation}
we have
\begin{equation}\label{g-b}
|\Fr(y)| \lesssim \tfrac{1}{\langle y\rangle}\qtq{for}y>0.
\end{equation}
\end{lemma} 
\begin{proof} In fact, \eqref{g-e} and \eqref{g-b} follow from contour integration.  Alternatively, one has \eqref{g-b} by writing
\begin{equation}\label{gaussian-ibp}
e^{-\frac{i\lambda}{2} x^2} = \frac{\partial_x\bigl[xe^{-\frac{i\lambda}{2}x^2}\bigr]}{1-i\lambda x^2}
\end{equation}
(with $\lambda=1$) and integrating by parts. \end{proof}

Our first result is the following: 

\begin{proposition}\label{P:infinity} For $\psi\in H^1(\R)$ and $t>0$, 
\[
\|[V(t)\psi](x) - T(|x|)\psi(x)-R(|x|)\und\psi( x)-2\Fr(\sqrt{t}|x|)\psi(0)\|_{L^\infty(\R)} \lesssim t^{-\frac14}\|\psi\|_{H^1(\R)},
\]
where $T(\cdot)$ and $R(\cdot)$ are as in \eqref{TR} and $\Fr(\cdot)$ is as in \eqref{Fr}. 
\end{proposition}

\begin{proof}[Proof of Proposition~\ref{P:infinity}]  Let $\psi\in H^1$ and $t>0$.  Write $[V(t)\psi](x)$ in the form \eqref{V1}.  We will consider only the case $x>0$, as the other case is similar. In particular, only the term containing $a_t(x,\xi)$ contributes, and we find
\begin{align}
[V(t)\psi](x) & = \sqrt{\tfrac{it}{2\pi}}\int_+ e^{-\frac{it}{2}(x-\xi)^2}T(\xi)\psi(\xi)\,d\xi \label{va1} \\
& \quad +  \sqrt{\tfrac{it}{2\pi}}\int_- e^{-\frac{it}{2}(x+\xi)^2}R(\und\xi)\psi(\xi)\,d\xi \label{va2} \\
& \quad +  \sqrt{\tfrac{it}{2\pi}}\int_- e^{-\frac{it}{2}(x-\xi)^2}\psi(\xi)\,d\xi.\label{va3} 
\end{align} 

We write
\begin{align}
\eqref{va1} & = T(x)\psi(x)\sqrt{\tfrac{it}{2\pi}}\int_+ e^{-\frac{it}{2}(\xi-x)^2}\,d\xi \label{va4} \\
&\quad  +  \sqrt{\tfrac{it}{2\pi}}\int_+ e^{-\frac{it}{2}(\xi-x)^2}[T(\xi)\psi(\xi)-T(x)\psi(x)]\,d\xi. \label{va5}
\end{align}
To proceed, we make the change variables $\eta=t^{\frac12}(\xi-x)$ in the integral for \eqref{va4} and use \eqref{g-e} and \eqref{Fr}: 
\begin{align*}
\eqref{va4} &= T(x)\psi(x)\sqrt{\tfrac{i}{2\pi}}\biggl[\int_\R e^{-\frac{i\eta^2}{2}}\,d\eta - \int_{-\infty}^{-x\sqrt{t}}e^{-\frac{i\eta^2}{2}}\,d\eta\biggr] \\
& = T(x)\psi(x) - T(0)\psi(0)\Fr(x\sqrt{t}) - [T(x)\psi(x)-T(0)\psi(0)]\Fr(x\sqrt{t}).
\end{align*}
Now note that $\partial_\xi T(\xi) = \frac{-iq}{(i\xi-q)^2}$ is bounded, so that
\begin{equation}\label{le1}
\|T\psi\|_{\dot H^1(\R)}\lesssim \|\psi\|_{H^1(\R)}.
\end{equation}
Thus, using \eqref{g-b} and and Cauchy--Schwarz,
\[
|T(x)\psi(x)-T(0)\psi(0)|\,|\Fr(x\sqrt{t})| \lesssim |x|^{\frac12}\langle \sqrt{t} x\rangle^{-1}\|\psi\|_{H^1} \lesssim t^{-\frac14}\|\psi\|_{H^1}.
\]
Noting that $T(0)=0$, we may continue from above to deduce
\[
\eqref{va4}=T(x)\psi(x) + \mathcal{O}(t^{-\frac14}\|\psi\|_{H^1(\R)}),
\]
giving a desired contribution plus an acceptable error term. 

We regard \eqref{va5} as an error term.  Using the identity \eqref{gaussian-ibp} (with $\lambda=t$) and integrating by parts, we compute
\begin{align}
\eqref{va5} & = -\sqrt{\tfrac{it}{2\pi}}e^{-\frac{itx^2}{2}}\tfrac{x}{1-itx^2}[T(x)\psi(x)-T(0)\psi(0)] \label{va6}\\
& \quad -\sqrt{\tfrac{it}{2\pi}}\int_+ e^{-\frac{it}{2}(\xi-x)^2}\tfrac{\xi-x}{1-it(\xi-x)^2}\partial_\xi[T(\xi)\psi(\xi)]\,d\xi \label{va7}\\
& \quad -\sqrt{\tfrac{it}{2\pi}}\int_+ e^{-\frac{it}{2}(\xi- x)^2}\tfrac{2it(\xi-x)}{[1-it(\xi-x)^2]^2} [T(\xi)\psi(\xi)-T(x)\psi(x)]\,d\xi.\label{va8}
\end{align}
We first recall \eqref{le1} and use Cauchy--Schwarz to estimate 
\[
|\eqref{va6}| \lesssim \tfrac{t^{\frac12}|x|^{\frac32}}{1+tx^2}\|\psi\|_{H^1(\R)}\lesssim t^{-\frac14}\|\psi\|_{H^1(\R)},
\]
which is acceptable.  For \eqref{va7},  we estimate by Cauchy--Schwarz, using the change of variables $\eta=t^{\frac12}(\xi-x)$ in the remaining integral; this leads to
\[
|\eqref{va7}| \lesssim t^{-\frac14}\|\psi\|_{H^1} \biggl(\int \frac{\eta^2}{1+\eta^4}\,d\eta\biggr)\lesssim t^{-\frac14}\|\psi\|_{H^1},
\]
which is acceptable.  For \eqref{va8}, we first bound
\[
|T(x)\psi(x)-T(\xi)\psi(\xi)| \lesssim \|\psi\|_{H^1}|x-\xi|^{\frac12} \qtq{and} \bigl| \tfrac{2it(\xi-x)^2}{(1-it(\xi-x)^2)^2}\bigl| \lesssim \tfrac{1}{1+t|\xi-\eta|^2}. 
\]
Thus, using the same change of variables,
\[
|\eqref{va8}| \lesssim t^{-\frac14}\|\psi\|_{H^1}\biggl(\int \frac{|\eta|^{\frac12}}{1+\eta^2}\,d\eta\biggr) \lesssim t^{-\frac14}\|\psi\|_{H^1},
\]
which is acceptable. 

We turn to \eqref{va2}. By a change of variables we have
\[
\eqref{va2} = \sqrt{\tfrac{it}{2\pi}}\int_+ e^{-\frac{it}{2}(x-\xi)^2}R(\xi)\psi(\und\xi)\,d\xi,
\]
putting this term in the same form as \eqref{va1}, with $T$ replaced by $R$ and $\psi(\cdot)$ replaced by $\psi(\und{\cdot})$.  Noting that $\partial_\xi R = \partial_\xi T$, we use the exact same analysis as above to show that
\[
\eqref{va2} = R(x)\psi(\und x)  +\Fr(x\sqrt{t})\psi(0) +\mathcal{O}(t^{-\frac14}\|\psi\|_{H^1}),
\]
giving a desired contribution plus an acceptable error term.  Here we have used $R(0)=-1$ (as opposed to $T(0)=0$).

It remains to consider \eqref{va3}. In fact, this term can be treated like \eqref{va1} and \eqref{va2}.  We write
\begin{align}
\eqref{va3} & = \psi(0)\sqrt{\tfrac{it}{2\pi}}\int_- e^{-\frac{it}{2}(x-\xi)^2}\,d\xi\label{va8.5} \\
& \quad+ [\psi(x)-\psi(0)]\sqrt{\tfrac{it}{2\pi}}\int_- e^{-\frac{it}{2}(x-\xi)^2}\,d\xi \label{va9} \\
& \quad + \sqrt{\tfrac{it}{2\pi}}\int_- e^{-\frac{it}{2}(x-\xi)^2}[\psi(\xi)-\psi(x)]\,d\xi. \label{va10} 
\end{align}
For the first term, we have
\[
\eqref{va8.5} = \Fr(x\sqrt{t})\psi(0),
\]
which is an acceptable contribution.  Changing variables and using \eqref{g-b} as above, we can estimate
\[
|\eqref{va9}| \lesssim |x|^{\frac12}\langle x\sqrt{t}\rangle^{-1}\|\psi\|_{H^1} \lesssim t^{-\frac14}\|\psi\|_{H^1}, 
\]
which is acceptable. Similarly, estimating as we did for \eqref{va5} (using \eqref{gaussian-ibp} and integrating by parts), we find
\[
|\eqref{va10}| \lesssim t^{-\frac14}\|\psi\|_{H^1},
\]
which is acceptable.  \end{proof}

The following is an immediate consequence of Proposition~\ref{P:infinity}. 

\begin{corollary}\label{C:infinity} Let $\psi\in H^1(\R)$ and $t>0$.  Then
\begin{align*}
|[V(t)\psi](0)| & \lesssim |\psi(0)| + t^{-\frac14}\|\psi\|_{H^1(\R)}, \\
\|V(t)\psi\|_{L^\infty(\R)}& \lesssim \|\psi\|_{L^\infty(\R)} + t^{-\frac14}\|\psi\|_{H^1(\R)}. 
\end{align*}
\end{corollary}

We turn to $V^{-1}(t)$.  Let $t>0$.  Proceeding as in \eqref{at}, we use the first expression in \eqref{ke2} and introduce the function $b_t(x,\xi)$ via
\begin{align*}
&\overline{\sqrt{\tfrac{it}{2\pi}}e^{-\frac{it}{2}(x^2+k^2)}K(tx,\xi)} \\
&=2\sqrt{\tfrac{t}{2\pi i}}P_e\bigl[\1(\xi)\bigl\{e^{\frac{it}{2}(x-\xi)^2}[\1(x)T(\und\xi)+1(\und x)] + e^{\frac{it}{2}(x+\xi)^2}\1(\und x)R(\und\xi)\bigr\}\bigr] \\
& =: 2P_e [b_t(x,\xi)]
\end{align*}
It then follows that
\begin{equation}\label{V2}
[V^{-1}(t)\phi](\xi) = \int_\R b_t(x,\xi)\phi(x)\,dx + \int_\R b_t(x,\und\xi)\phi(\und x)\,dx. 
\end{equation}

\begin{proposition}\label{P:infinity2} For $\phi\in H^1(\R)$ and $t>0$, 
\[
\|[V^{-1}(t)\phi](\xi)-\bar T(|\xi|)\phi(\xi)-\bar R(|\xi|)\und\phi(\xi)-2\overline{\Fr}(\sqrt{t}|\xi|)\phi(0)\|_{L^\infty(\R)}\lesssim t^{-\frac14}\|\phi\|_{H^1(\R)},
\]
where $T(\cdot)$ and $R(\cdot)$ are as in \eqref{TR} and $\Fr(\cdot)$ is as in \eqref{Fr}. 
\end{proposition}

\begin{proof} The proof is similar to the proof of Proposition~\ref{P:infinity2}, so we only sketch the main points.  We consider the case $\xi>0$, in which case
\begin{align}
[V^{-1}\phi](\xi) & = T(\und\xi)\sqrt{\tfrac{t}{2\pi i}}\int_+ e^{\frac{it}{2}(x-\xi)^2}\phi(x)\,dx \label{via1} \\
&\quad + R(\und\xi)\sqrt{\tfrac{t}{2\pi i}}\int_- e^{\frac{it}{2}(x+\xi)^2} \phi(x)\,dx \label{via2} \\
&\quad+\sqrt{\tfrac{t}{2\pi i}}\int_- e^{\frac{it}{2}(x-\xi)^2}\phi(x)\,dx.  \label{via3}
\end{align}
We write
\begin{align}
\eqref{via1} & = T(\und\xi)\phi(\xi)\sqrt{\tfrac{t}{2\pi i}}\int_+ e^{\frac{it}{2}(x-\xi)^2}\,dx \label{via4} \\
& \quad + T(\und\xi)\sqrt{\tfrac{t}{2\pi i}}\int_+ e^{\frac{it}{2}(x-\xi)^2}[\phi(x)-\phi(\xi)]\,dx\label{via5}. 
\end{align}
Then
\[
\eqref{via4} = T(\und\xi)\phi(\xi)-\overline{\Fr}(\sqrt{t}\xi)T(0)\phi(0)-\overline{\Fr}(\sqrt{t}\xi)[T(\xi)\phi(\xi)-T(0)\phi(0)].
\]
The first term gives an acceptable contribution; the second term is zero since $T(0)=0$; and using \eqref{g-b} and Cauchy--Schwarz, the third term is controlled by
\[
|\xi|^{\frac12}\langle\sqrt{t}\xi\rangle^{-1}\|T\phi\|_{\dot H^1}\lesssim t^{-\frac14}\|\phi\|_{H^1}, 
\]
which is acceptable.  The term \eqref{via5} is amenable to the same analysis used to estimate \eqref{va5} above; in particular, using the identity \eqref{gaussian-ibp} (with $\lambda=-t$) and integrating by parts, one is led to the estimate
\[
|\eqref{via5}| \lesssim t^{-\frac14}\|\phi\|_{H^1},
\]
which is acceptable. 

Next, note that a change of variables puts \eqref{via2} into the same form as \eqref{via1}, with $T$ replaced by $R$ and $\phi(\cdot)$ replaced by $\phi(\und\cdot)$.  Thus
\[
\eqref{via2} = R(\und\xi)\phi(\und\xi)+ \overline{\Fr}(\sqrt{t}\xi)\phi(0)+\mathcal{O}(t^{-\frac14}\|\phi\|_{H^1}),
\]
which is acceptable.  Here we have used $R(0)=-1$.  

Finally, for \eqref{via3} note that
\[
\sqrt{\tfrac{t}{2\pi i}}\int_- e^{\frac{it}{2}(x-\xi)^2}\phi(x)\,dx = \overline{\sqrt{\tfrac{it}{2\pi}}\int_- e^{-\frac{it}{2}(x-\xi)^2}\bar\phi(x)\,dx}.
\]
Thus, arguing exactly as for \eqref{va3}, we find
\[
\eqref{via3} = \overline{\Fr}(\sqrt{t}\xi)\phi(0) + \mathcal{O}(t^{-\frac14}\|\phi\|_{H^1}),
\]
which is acceptable. 
\end{proof}

We next establish some $\dot H^1$-bounds for $V$ and $V^{-1}$. 

\begin{proposition}\label{P:H1} Let $\psi\in H^1(\R)$ satisfy $\psi(0)=0$. Then for 
any $t>0$, we have 
\[
\|V(t)\psi\|_{\dot H^1(\R)} \lesssim 
\|\psi\|_{H^1(\R)}. 
\]
\end{proposition}

\begin{proof} As in the proof of Proposition~\ref{P:infinity}, we write $[V(t)\psi](x)$ in the form \eqref{V1} and focus on the case $\xi>0$, so that 
\[
[V(t)\psi](x)=\eqref{va1}+\eqref{va2}+\eqref{va3}.
\] 

Writing 
\begin{equation}\label{agi}
\partial_x e^{-\frac{it}{2}(x-\xi)^2}=-\partial_\xi e^{-\frac{it}{2}(x-\xi)^2}
\end{equation}
and integrating by parts, we find
\begin{align}
\partial_x\bigl[\eqref{va1}\bigr] & = \sqrt{\tfrac{it}{2\pi}}e^{-\frac{it}{2}x^2}T(0)\psi(0) \label{vh1} \\
& \quad + \sqrt{\tfrac{it}{2\pi}}\int_+ e^{-\frac{it}{2}(\xi-x)^2}\partial_\xi\bigl[T(\xi)\psi(\xi)\bigr]\,d\xi. \label{vh2}
\end{align}
Note that $T(0)=0$ implies $\eqref{vh1}=0$.  As for \eqref{vh2}, we may write
\[
\eqref{vh2}=\sqrt{it}e^{-\frac{it}{2}x^2}\F_0^{-1}\bigl[\1(\xi)e^{-\frac{it}{2}\xi^2}\partial_\xi\bigl(T(\xi)\psi(\xi)\bigr)\bigr](tx),
\]
whence by Plancherel, 
\[
\|\eqref{vh2}\|_{L_x^2}\lesssim \|T\psi\|_{\dot H^1}\lesssim \|\psi\|_{H^1}. 
\]
For the contribution of \eqref{va2}, we argue similarly to deduce
\[
\|\partial_x \eqref{va2}\|_{L_x^2} \lesssim t^{\frac12} |\psi(0)| + 
\|\psi\|_{H^1}\lesssim \|\psi\|_{H^1},
\]
where we used $\psi(0)=0$.  Finally, \eqref{va3} is treated similarly to give
\[
\|\partial_x \eqref{va3}\|_{L_x^2} \lesssim 
\|\psi\|_{\dot H^1}. 
\]
This completes the proof. \end{proof}

\begin{proposition}\label{P:H12} For $t>0$ and $\phi\in H^1(\R)$,
\[
\|V^{-1}(t)\phi\|_{\dot H^1(\R)}\lesssim t^{\frac12}|\phi(0)| + \|\phi\|_{H^1(\R)}. 
\]
\end{proposition}

\begin{proof}  As in the proof of Proposition~\ref{P:infinity2}, we write $[V^{-1}(t)\phi](\xi)$ in the form \eqref{V2} and focus on the case $\xi>0$, so that
\[
[V^{-1}(t)\phi](\xi) = \eqref{via1}+\eqref{via2}+\eqref{via3}. 
\]
Applying $\partial_\xi$, we are led to the terms
\begin{align}
&\partial_\xi[T(\und\xi)]\sqrt{\tfrac{t}{2\pi i}}\int_+ e^{\frac{it}{2}(x-\xi)^2}\phi(x)\,dx, \label{vh3} \\
&\partial_\xi[R(\und\xi)]\sqrt{\tfrac{t}{2\pi i}}\int_- e^{\frac{it}{2}(x+\xi)^2} \phi(x)\,dx, \label{vh4}
\end{align}
and three more terms similar to the ones appearing in the proof of Proposition~\ref{P:H1} (cf. \eqref{agi}).  In particular, it suffices to estimate \eqref{vh3} and \eqref{vh4}.  However, as $\partial_\xi T = \partial_\xi R$ is bounded, we can estimate 
\[
\|\eqref{vh3}\|_{L^2} + \|\eqref{vh4}\|_{L^2} \lesssim \|\phi\|_{L^2}
\]
by arguing as we did for \eqref{vh2}, say.  This completes the proof.
\end{proof}

We record the following useful consequence of Proposition~\ref{P:H12}:

\begin{corollary}\label{C:H1} For $\phi\in H^1(\R)$ and $t>0$, we have $[V^{-1}(t)\phi](0)=0$. 
\end{corollary}

\begin{proof} Proposition~\ref{P:H12} shows that $V^{-1}(t)\phi \in H^1(\R)\hookrightarrow C(\R)$, so that $[V^{-1}(t)\phi](0)$ is well-defined.  Using the identity in Lemma~\ref{V-ik} and the fact that $K(\cdot,0)\equiv 0$, one can deduce the result.\end{proof}


\section{Global existence and decay}\label{S:decay} In this section, we prove that small initial data in $\Sigma$ lead to global decaying solutions.  In particular, we will establish the estimate \eqref{E:decay} appearing in Theorem~\ref{T}.  In fact, one has a global solution $u\in C(\R;H^1)$ for any $u(0)\in H^1$ by the standard $H^1$ well-posedness theory and the conservation of mass and energy (see \cite{FOO}, for example).  Thus, it suffices to prove \emph{a priori} estimates for solutions.   Using conservation of mass and energy, it is straightforward to prove that if $\|u_0\|_{H^1}=\eps\ll1$, then 
\begin{equation}\label{uH1}
\sup_{t\in[0,\infty)}\|u(t)\|_{H^1(\R)} \lesssim\eps,
\end{equation}
even in the focusing case.

We will use the variables $w(t)$ described in the introduction, namely
\begin{equation}\label{w}
w(t) = \F_q U(-t)u(t),\qtq{or equivalently} u(t)=M(t)D(t)V(t)w(t).
\end{equation}
If $u$ is a solution to \eqref{nls}, then $w$ solves the equation
\begin{equation}\label{weq}
i\partial_t w = \lambda \F_q U(-t)\bigl(|u|^2 u\bigr) = \lambda t^{-1}V(t)^{-1}\bigl[|V(t)w|^2 V(t)w|^2\bigr]. 
\end{equation}

In addition to proving $L^\infty$ bounds for $w$, we will need to estimate the $\dot H^1$-norm of $w$.  This is analogous to controlling the $L^2$-norm of $(x+it\partial_x)u$ in the case of the free NLS.  We begin with two nonlinear estimates that will aid in this task.

\begin{lemma}[Nonlinear estimates]\label{L:NLE} The following bounds hold:

Suppose $u\in \Sigma$ and $t\in[0,1]$.  Then
\begin{equation}\label{E:NLE1}
\|\F_q U(-t)\bigl(|u|^2 u\bigr)\|_{\dot H^1} \lesssim \|u\|_{H^1}^3 + \|u\|_{H^1}^2 \|\langle x\rangle u\|_{L^2}.
\end{equation}

Suppose $w\in H^1$ and $w(0)=0$.  For $t\geq1$,
\begin{equation}\label{E:NLE2}
\|V(t)^{-1}\bigl[|V(t)w|^2 V(t)w\bigr]\|_{\dot H^1} \lesssim t^{-\frac14}\|w\|_{H^1}^3 +\bigl[\|w\|_{L^\infty} + t^{-\frac14}\|w\|_{H^1}\bigr]^2\|w\|_{H^1}.
\end{equation}
\end{lemma}

\begin{proof} We first compute
\[
\partial_\xi\F_q U(-t) = e^{\frac{it}{2}\xi^2}[\partial_\xi + it\xi]\F_q.
\]
Thus, using \eqref{E:DF} and \eqref{E:XF}, we have the following estimate for any $t\in[0,1]$:
\begin{align*}
\|\partial_\xi\F_q U(-t)\bigl(|u|^2 u\bigr)\|_{L^2} & \lesssim \|\langle x\rangle|u|^2 u\|_{L^2} + \| |u|^2 u\|_{ H^1}  \\
& \lesssim \|u\|_{L^\infty}^2 \|\langle x\rangle u\|_{L^2} +\|u\|_{L^\infty}^2 \|u\|_{H^1},
\end{align*}
which by Sobolev embedding gives \eqref{E:NLE1}.  

We turn to \eqref{E:NLE2}.  Using Proposition~\ref{P:H12},  Corollary~\ref{C:infinity}, Proposition~\ref{P:H1}, and $w(0)=0$, we estimate
\begin{align*}
\|V^{-1}\bigl(|Vw|^2 Vw\bigr)\|_{\dot H^1} & \lesssim \| |Vw|^2 Vw\|_{H^1} + t^{\frac12} \bigl|(Vw)|_{x=0}\bigr|^3 \\
& \lesssim \|Vw\|_{L^\infty}^2 \|Vw\|_{H^1} + t^{-\frac14}\|w\|_{H^1}^3 \\
& \lesssim \bigl[\|w\|_{L^\infty} + t^{-\frac14}\|w\|_{H^1}\bigr]^2 \|w\|_{H^1}+t^{-\frac14}\|w\|_{H^1}^3,
\end{align*}
which is \eqref{E:NLE2}.
\end{proof}

We first show that we can propagate bounds for a short time.

\begin{proposition}[Short-time bounds]\label{P:stb}   Let $\|u_0\|_\Sigma=\eps$ and let $u$ denote the solution to \eqref{nls} with $u(0)=u_0$.  For $\eps$ sufficiently small, the function 
\[
w(t)=\F_q U(-t)u(t)
\]
satisfies the following bounds:
\[
\sup_{t\in[0,1]} \|w(t)\|_{H^1(\R)}\lesssim \eps. 
\]
\end{proposition}

\begin{proof} Using \eqref{weq}, we first write
\[
\partial_\xi w(t) = \partial_\xi\F_q u_0 -i\lambda\int_0^t \partial_\xi \F_q U(-s)\bigl(|u|^2 u\bigr)(s)\,ds. 
\]
Using \eqref{E:DF}, we note
\[
\|\partial_\xi\F_q u_0\|_{L^2} \lesssim \|u_0\|_{\Sigma}\lesssim \eps. 
\]

For the nonlinear term, we use \eqref{E:NLE1} to bound
\[
\|\partial_\xi\F_q U(-s)\bigl(|u|^2 u\bigr)(s)\|_{L^2}\lesssim \|u(s)\|_{H^1}^3 + \|u(s)\|_{H^1}^2\|\langle x\rangle u(s)\|_{L^2}.
\]
To relate $u$ back to $w$, we first write
\[
 xu(s)=xU(s)\F_q^{-1}w(s)=x\F_q^{-1} e^{-\frac{is}{2}\xi^2}w(s).
\]
We can therefore use \eqref{E:XF1}, the fact that $w(s,0)=0$ (cf. \eqref{E:F00}), and \eqref{E:XF} to bound
\begin{align*}
\|xu(s)\|_{L^2} \lesssim \| e^{-\frac{is}{2}\xi^2} w(s)\|_{H^1} &\lesssim \|w(s)\|_{H^1} + s\|\xi w(s)\|_{L^2}\\ &\lesssim \|w(s)\|_{H^1}+s\|u(s)\|_{H^1}.
\end{align*}

It follows that for $t\in[0,1]$ and $\eps$ small, we have
\[
\|w(t)\|_{H^1} \lesssim \eps + \int_0^t \eps^2 \|w(s)\|_{H^1}\,ds.
\]
An application of Gronwall's inequality now completes the proof. \end{proof}

We turn to proving long-time bounds for the function $w$ defined in \eqref{w}.  The key will be to study the differential equation satisfied by $w$, namely 
\begin{equation}\tag{\ref{ode}}
i\partial_t w = \lambda t^{-1} V(t)^{-1}\bigl[|V(t)w|^2 V(t)w\bigr]. 
\end{equation}
In fact, because the asymptotics for $V$ and $V^{-1}$ involve the input function as well as its reflection (cf. Propositions~\ref{P:infinity}~and~\ref{P:infinity2}), we encounter a $2\times 2$ system satisfied by ${ }^t(w,\und w$).  To proceed, we introduce the notation
\[
\vec f = { }^t(f,\und f),\quad \und{\vec f}={ }^t(\und f, f).
\]

\begin{proposition}[Approximate equation for $w$]\label{P:ode} Let $u$ be a solution to \eqref{nls} and $w(t)=\F_q U(-t)u(t).$  For $t\geq 1$, 
\[
i\partial_t \vec w(t) = \lambda t^{-1} A(t)\vec w(t)+ \mathcal{O}\bigl( t^{-\frac54}\bigl[\|w(t)\|_{L^\infty} + t^{-\frac14}\|w(t)\|_{H^1}\bigr]^2\|w(t)\|_{H^1}\bigr)
\]
in $L^\infty(\R)$, where $A(t)=A(\vec w(t,x),x)$ is the $2\times2$ hermitian matrix defined as follows: writing
\[
\mathbf{S}(x)=(S_1(x),S_2(x)):=(T(|x|),R(|x|)), 
\]
where $T$ and $R$ are as in \eqref{TR}, the entries $A_{ij}$ of $A$ are:
\begin{align*}
A_{11} & = A_{11}^1+ A_{11}^2 := |\mathbf{S}\cdot \vec w|^2|S_1|^2+|\mathbf{S}\cdot \und{\vec w}|^2|S_2|^2, \\
A_{12}=\overline{A_{21}} & =A_{12}^1 + A_{12}^2 := |\mathbf{S}\cdot \vec w|^2  \bar S_1 S_2 + |\mathbf{S}\cdot \und{\vec w}|^2\bar S_2 S_1, \\
A_{22} & = A_{22}^1 + A_{22}^2 := |\mathbf{S}\cdot \vec w|^2|S_2|^2 + |\mathbf{S} \cdot \und{\vec w}|^2|S_1|^2. 
\end{align*}
\end{proposition}

\begin{proof} To begin, recall that
\[
w(t) = \F_q U(-t) u(t),
\]
which by \eqref{E:F00} implies that $w(t,0)\equiv 0$.  Thus, using Corollary~\ref{C:infinity} and Proposition~\ref{P:H1}, we have
\begin{equation}\label{abds}
\begin{aligned}
\bigl|[V(t)w(t)]_{x=0}\bigr| & \lesssim t^{-\frac14}\|w(t)\|_{H^1}, \\
\|V(t)w(t)\|_{L^\infty} & \lesssim \|w(t)\|_{L^\infty} + t^{-\frac14}\|w(t)\|_{H^1}, \\
\|V(t)w(t)\|_{\dot H^1} & \lesssim \|w(t)\|_{H^1}.
\end{aligned}
\end{equation}
Applying Proposition~\ref{P:infinity2} and \eqref{abds}, we therefore have
\begin{align*}
V^{-1}[|Vw|^2 Vw] & = \bar S_1 |Vw|^2Vw + \bar S_2 |\und{Vw}|^2 \und{Vw} + \mathcal{O}(t^{-\frac14}\| |Vw|^2 Vw\|_{\dot H^1}\!+ t^{-\frac34}\|w\|_{H^1}^3\bigr) \\
& =  \bar S_1 |Vw|^2Vw + \bar S_2 |\und{Vw}|^2 \und{Vw} + \mathcal{O}\bigl(t^{-\frac14}\bigl[\|w\|_{L^\infty}+t^{-\frac14}\|w\|_{H^1}\bigr]^2\|w\|_{H^1}\bigr).
\end{align*}

Next, note that Proposition~\ref{P:infinity} and $w(t,0)=0$ implies
\[
\|Vw - \mathbf{S} \cdot\vec w \|_{L^\infty} \lesssim t^{-\frac14}\|w\|_{H^1}. 
\]
Recalling \eqref{abds} and the definition of $A_{11}^1$ and $A_{12}^1$ above, we now write
\begin{align*}
\bar S_1 |Vw|^2 Vw & = \bar S_1|\mathbf{S}\cdot \vec w|^2 \mathbf{S}\cdot \vec w + \mathcal{O}\bigl(t^{-\frac14}\bigl[\|w\|_{L^\infty} + t^{-\frac14}\|w\|_{H^1}\bigr]^2 \|w\|_{H^1}\bigr) \\
& = A_{11}^1 w + A_{12}^1 \und w + 
\mathcal{O}\big(t^{-\frac14}\bigl[\|w\|_{L^\infty} + t^{-\frac14}\|w\|_{H^1}\bigr]^2\|w\|_{H^1}\big), 
\end{align*}
which is acceptable. 

We turn to the term $\bar S_2|\und{Vw}|^2 \und{Vw}$.  We proceed as above, using the fact that
\[
\und{\mathbf{S} \cdot \vec w} = \mathbf{S} \cdot \und{\vec w}.
\]
Thus
\begin{align*}
\bar S_2|\und{Vw}|^2 \und{Vw} &= \bar S_2 |\mathbf{S} \cdot \und{\vec w}|^2\mathbf{S}\cdot \und{\vec w} +  \mathcal{O}\bigl(t^{-\frac14}\bigl[\|w\|_{L^\infty} + t^{-\frac14}\|w\|_{H^1}\bigr]^2 \|w\|_{H^1}\bigr) \\
& = A_{11}^2 w + A_{12}^2 \und w +  \mathcal{O}\bigl(t^{-\frac14}\bigl[\|w\|_{L^\infty} + t^{-\frac14}\|w\|_{H^1}\bigr]^2 \|w\|_{H^1}\bigr),
\end{align*}
which is acceptable. 

Similar analysis applied to $\und w$ leads to the formulas for $A_{21}$ and $A_{22}$. \end{proof}

We are now in a position to prove long-time bounds for solutions to \eqref{nls}. 

\begin{proposition}[Long-time bounds]\label{P:ltb} Fix $\beta\in(0,\frac18)$.  Let $\|u_0\|_\Sigma=\eps$ and let $u$ denote the solution to \eqref{nls} with $u(0)=u_0$.  For $\eps$ sufficiently small, the function 
\[
w(t)=\F_q U(-t)u(t)
\]
satisfies the following bounds:
\begin{equation}\label{E:ltb}
\sup_{t\in[0,\infty)}\bigl\{ \|w(t)\|_{L^\infty} + \langle t\rangle^{-\beta}\|w(t)\|_{H^1}\bigr\}\lesssim \eps.
\end{equation}
Consequently, 
\begin{equation}\label{E:ltb2}
\sup_{t\in[0,\infty)} \langle t\rangle^{\frac12}\|u(t)\|_{L^\infty} \lesssim\eps. 
\end{equation}
\end{proposition}

\begin{proof} Let us first deduce \eqref{E:ltb2} from \eqref{E:ltb}.  As \eqref{uH1} implies
\[
\|u(t)\|_{L^\infty} \lesssim \|u(t)\|_{H^1}\lesssim\eps
\]
for all $t\in[0,\infty)$, it suffices to consider times $t\geq 1$. To this end, recall that
\[
u(t)=M(t)D(t)V(t)w(t),\qtq{so that}\|u(t)\|_{L^\infty} \lesssim t^{-\frac12}\|V(t)w(t)\|_{L^\infty}. 
\]
Now, recalling $w(t)=\F_q U(-t)u(t)$ (so that \eqref{E:F00} gives $w|_{x=0}=0$), Corollary~\ref{C:infinity} and \eqref{E:ltb} imply
\[
\|V(t)w(t)\|_{L^\infty} \lesssim \|w(t)\|_{L^\infty} + t^{-\frac14}\|w(t)\|_{H^1} \lesssim \eps
\]
for all $t\geq 1$.  Thus \eqref{E:ltb2} follows from \eqref{E:ltb}.

We turn to \eqref{E:ltb}.  In light of Proposition~\ref{P:stb} (and Sobolev embedding), it suffices to consider times $t\geq 1$.  To this end, we fix $T\geq 1$ and define
\[
\|w\|_{X_T} := \sup_{t\in[1,T]} \bigl\{ \|w(t)\|_{L^\infty} + t^{-\beta}\|w(t)\|_{H^1}\bigr\}.
\] 

We first use \eqref{ode} to write
\[
w(t) = w(1) - i\lambda \int_1^t s^{-1}V(s)^{-1}\bigl[|V(s)w(s)|^2V(s)w(s)\bigr]\,ds.
\]
As observed above, we have $w(t,0)\equiv 0$.  Hence we may apply the nonlinear estimate \eqref{E:NLE2} to bound
\begin{align*}
\|w(t)\|_{H^1} &\lesssim \eps + \int_1^t s^{-\frac54}\|w(s)\|_{H^1}^3 +s^{-1}\bigl[\|w(s)\|_{L^\infty} + s^{-\frac14}\|w(s)\|_{H^1}\bigr]^2\|w(s)\|_{H^1}\,ds \\
& \lesssim \eps + \bigl[t^{-\frac14+3\beta}+t^{\beta} + t^{-\frac12+3\beta}\bigr] \|w\|_{X_T}^3
\end{align*}
for $t\in[1,T]$.  Thus
\begin{equation}\label{bs1}
\|w(t)\|_{H^1}\lesssim  \eps + t^{\beta}\|w\|_{X_T}^3 \qtq{for any}t\in[1,T]. 
\end{equation}

Next, using Proposition~\ref{P:ode}, we may write
\[
i\partial_t \vec w(t) = \lambda t^{-1} A(t) \vec w(t) + \mathcal{O}(t^{-\frac54+\beta}\|w\|_{X_T}^3)\qtq{for}t\in[1,T].
\]
As $A(t)$ is hermitian, this implies
\[
\partial_t |w(t)|^2 =\mathcal{O}(t^{-\frac54+\beta}\|w\|_{X_T}^3)|w(t)|,
\]
whence
\begin{equation}\label{bs2}
\|w(t)\|_{L^\infty} \lesssim \eps + \|w\|_{X_T}^3 \qtq{for any}t\in[1,T]. 
\end{equation}

A standard continuity argument using \eqref{bs1} and \eqref{bs2} now shows
\[
\|w\|_{X_T} \lesssim \eps \qtq{for any}T\geq 1,
\]
where the implicit constant is independent of $T$.  This implies \eqref{E:ltb}.  \end{proof}

\section{Asymptotic behavior}\label{S:asymptotics} 

Finally, we turn to the question of the asymptotic behavior of solutions to \eqref{nls}.  In particular, we establish the asymptotics \eqref{E:asymptotics} in Theorem~\ref{T}.  As in the previous section, we will work with the variables $w(t)=\F_qU(-t)u(t)$ and study the behavior of solutions to the equation satisfied by $w$, namely \eqref{ode}. 

\begin{proposition}[Asymptotic behavior]\label{P:asymptotics} Fix $\beta\in(0,\frac18)$.  Let $\|u_0\|_\Sigma=\eps$ and let $u$ denote the solution to \eqref{nls} with $u(0)=u_0$.  For $\eps$ sufficiently small, the function 
\[
w(t)=\F_q U(-t)u(t)
\]
satisfies the following asymptotics: there exists unique $W\in L^\infty$ so that
\begin{equation}\label{E:vwa}
[V(t)w(t)](x) = e^{-i\lambda |W(x)|^2 \log t}W(x)  + \mathcal{O}(t^{-\frac14+\beta})
\end{equation}
in $L^\infty$ as $t\to\infty$.  Consequently, 
\[
u(t) = M(t)D(t)\bigl[  e^{-i\lambda |W|^2 \log t}W\bigr] + \mathcal{O}(t^{-\frac34+\beta})
\]
in $L^\infty$ as $t\to\infty$. 
\end{proposition}

\begin{proof} As $u(t)=M(t)D(t)V(t)w(t)$, it suffices to establish \eqref{E:vwa}.  We suppose that $\eps$ is chosen small enough that the decay estimates of Proposition~\ref{P:ltb} hold. 

Recall the approximate equation for $w$ in Proposition~\ref{P:ode}, as well as the notation introduced there.  A tedious but elementary computation using \eqref{TR} shows that the hermitian matrix $A=A(w(t,x),x)$ may be diagonalized by a unitary matrix $B=B(x)$ as follows:
\[
B^*AB = \text{diag}\bigl(|\mathbf{S}\cdot \vec w|^2, |\mathbf{S}\cdot \und{\vec w}|^2\bigr),\quad B=\biggl(\begin{array}{cc} S_1 & \bar S_2 \\ -S_2 & \bar S_1\end{array}\biggr),
\]
where ${ }^*$ denotes conjugate transpose. 

Now introduce the variables $\mathbf{f} = { }^t(f_1,f_2) = B^*\vec w$. Noting that $B\mathbf{f}=\vec w$, we find
\[
|\mathbf{S}\cdot \vec w|^2 = |(S_1^2-S_2^2)f_1 + (S_1 \bar S_2 + \bar S_1 S_2)f_2|^2 = |f_1|^2,
\]
and similarly $|\mathbf{S}\cdot\und{\vec w}|^2 = |f_2|^2$.  In particular, recalling \eqref{E:ltb}, we find that $\mathbf{f}$ satisfies the diagonal system
\[
i\partial_t f_j = \lambda t^{-1} |f_j|^2 f_j + \mathcal{O}(t^{-\frac54+\beta}),\quad j\in\{1,2\}.
\]
Now the situation is similar to that of the free NLS (cf. \cite{HN1}): defining $g_j$ via  
\begin{equation}\label{gjfj}
g_j(t) := \exp\biggl\{i\lambda \int_1^t |f_j(s)|^2 \tfrac{ds}{s}\biggr\}f_j(t),  
\end{equation}
we have that
\begin{equation}\label{dtgj}
i\partial_t g_j(t) = \mathcal{O}(t^{-\frac54+\beta}) \qtq{and} |g_j| \equiv |f_j|. 
\end{equation}
It follows that $g_j(t)$ converges in $L^\infty$ at a rate of $t^{-\frac14+\beta}$ as $t\to\infty$.  In fact, using \eqref{gjfj} and \eqref{dtgj}, we can write
\[
f_j(t) = G_j(t) + \mathcal{O}(t^{-\frac14+\beta+}), \qtq{where} G_j(t) := e^{-i\lambda |\phi_j|^2 \log t} \phi_j
\]
for some $\phi_j\in L^\infty$.  In particular, $\phi_j$ equals the limit of $g_j(t)$ up to multiplication by some phase factor. 
%
%

Recalling $\vec w = B\mathbf f$, we find
\begin{align*}
w(t) &= S_1 G_1(t) + \bar S_2 G_2(t) + \mathcal{O}(t^{-\frac14+\beta+}), \\
\und w(t) &= -S_2 G_1(t) + \bar S_1 G_2(t) + \mathcal{O}(t^{-\frac14+\beta+}). 
\end{align*}
Using $w(t,0)\equiv 0$ (cf. \eqref{E:F00}), Proposition~\ref{P:infinity}, \eqref{E:ltb}, and \eqref{TR}, we deduce
\begin{align*}
V(t)w(t) & = S_1 w(t) + S_2 \und w(t) + \mathcal{O}(t^{-\frac14+\beta}) \\
& = (S_1^2-S_2^2)G_1(t) + (S_1\bar S_2 + S_2 \bar S_1)G_2(t) + \mathcal{O}(t^{-\frac14+\beta+}) \\
& = (S_1+S_2)G_1(t) + \mathcal{O}(t^{-\frac14+\beta+}). 
\end{align*}
Noting that $|S_1+S_2|=1$, we find that \eqref{E:vwa} holds with $W:=(S_1+S_2)\phi_1$.  This completes the proof. 
\end{proof}

%
%


\end{document}